\documentclass{article}

\usepackage{amssymb}
\usepackage{amsthm}
\usepackage{graphics}
\usepackage{setspace}
\usepackage{amsmath}
\usepackage{mathrsfs}
\usepackage{textcomp}
\usepackage{wasysym}
\usepackage{helvet}
\usepackage{enumerate} 
\usepackage{pifont}

\renewcommand{\geq}{\geqslant}		
\renewcommand{\leq}{\leqslant}		

\newtheorem{theorem}{Theorem}[section]			
\newtheorem{corollary}[theorem]{Corollary}		
\newtheorem{lemma}[theorem]{Lemma}			
\newtheorem{prop}[theorem]{Proposition}			
\newtheorem{newdef}{Definition}[section]

\def\eref#1{$(\ref{#1})$}				
\def\lref#1{Lemma~$\ref{#1}$}		
\def\dref#1{Definition~$\ref{#1}$}		
\def\tref#1{Theorem~$\ref{#1}$}		
\def\cref#1{Corollary~$\ref{#1}$}		
\def\pref#1{proposition~$\ref{#1}$}

\title{A Classification of Permutation Polynomials through some linear maps}
\author{Megha M. Kolhekar, Harish K. Pillai\\
  \small \{meghakolhekar, hp\}@ee.iitb.ac.in\\
  \small Department of Electrical Engineering, IIT Bombay\\
  \small India
}

\begin{document}
\maketitle

\abstract{In this paper, we propose linear maps over the space of all polynomials $f(x)$ in $\mathbb{F}_q[x]$ that map $0$ to itself, through their evaluation map. Properties of these linear maps throw up interesting connections with permutation polynomials. We study certain properties of these linear maps. We propose to classify permutation polynomials by identifying the generalized eigenspaces of these maps, where the permutation polynomials reside. As it turns out, several classes of permutation polynomials studied in literature neatly fall into classes defined using these linear maps. We characterize the shapes of permutation polynomials that appear in the various generalized eigenspaces of these linear maps. For the case of $\mathbb{F}_p$, these generalized eigenspaces provide a degree-wise distribution of polynomials (and therefore permutation polynomials) over $\mathbb{F}_p$.
  We show that for $\mathbb{F}_q$, it is sufficient to consider only a few of these linear maps. The intersection of the generalized eigenspaces of these linear maps contain (permutation) polynomials of certain shapes. In this context, we study a class of permutation polynomials over $\mathbb{F}_{p^2}$. We show that the permutation polynomials in this class are closed under compositional inverses. We also do some enumeration of permutation polynomials of certain shapes.}

\section{Introduction and Related Work}
Permutation polynomials (PPs) are important in cryptography, combinatorics and coding theory, to name a few areas. They have been studied extensively by various researchers for their properties, testing methods, constructions, cycle structures and compositional inverses. A good reference for the theory of permutation polynomials is \cite{Mullen:2013:HFF} while a list of relevant problems related to PPs can be found in the seminal paper \cite{Lidl:1988:whendoes}. \cite{Hou:2015:survey} is a good survey paper on the recent results related to permutation polynomials. \\
There has been consistent effort to identify new classes of PPs which are `easy' to construct. In \cite{Akbary08onsome}, the authors investigate permutation properties of polynomials $P(x)=x^r+x^{r+s}+\ldots+x^{r+ks}$ for $0<r<q-1, 0<s<q-1, k\geq 0$. They construct several classes of this form over $\mathbb{F}_q$ and enumerate them. In \cite{Yuan2008PermutationPO}, permutation polynomials of the form $(x^p-x+\delta)^s+L(x)$ over fields of characteristic $2$ and $3$ are proposed. In \cite{Zheng2013}, two new classes of permutation polynomials over finite fields are presented. Based on the link between equivalent equations and permutation polynomials, some other new permutation polynomials are also obtained. In \cite{Zheng:2016:LCP:2995737.2995773}, authors propose PPs of the form $f(x)=(ax^q+bx+c)^r\phi((ax^q+bx+c)^{(q^2-1)/d})+ux^q+vx$ over $\mathbb{F}_{q^2}$. They employ the `Akbary-Ghioca-Wang (AGW)' criterion two times, and reduce the problem of determining whether $f(x)$ permutes $\mathbb{F}_{q^2}$  to that of verifying whether two more polynomials permute two subsets of $\mathbb{F}_{q^2}$. They further unify these results and generalize some known classes of PPs. \cite{Gupta2016NCP} describes four new classes of permutation trinomials of the form $x^rh(x^{2^m-1})$ over $\mathbb{F}_{2^{2m}}$. In \cite{Wang2018}, the authors present six new classes of permutation trinomials over $\mathbb{F}_{2^n}$ which have explicit forms by determining the solutions of some equations. Some of the ideas of constructing new PPs are, for example, through additive cyclotomy \cite{2008arXiv0810.2830Z}, using the AWG criterion \cite{AKBARY201151} or through linearized polynomials \cite{evans1992},  \cite{Zheng:2016:LCP:2995737.2995773}. Determining the compositional inverse of a given permutation polynomial is also a challenging task. Very few papers such as \cite{TUXANIDY2014244} give compositional inverses of proposed classes. \\
In this paper,  we study linear maps over the space of polynomials that map $0$ to $0$ through their evaluation over $\mathbb{F}_q$. These maps have some specific properties that we explore. We study the generalized eigenspaces of these maps and define classes of permutation polynomials based on the generalized eigenspaces to which they belong. Therefore, the shapes of polynomials in the generalized eigenspaces of this map are determined. Especially, over $\mathbb{F}_{p^2}$ we come across a known class studied in \cite{Zheng:2016:LCP:2995737.2995773} which is a consequence of the specific shapes of polynomials appearing in the eigenspaces of these maps. We prove that the compositional inverses belong to the same class and give a parametric representation of the compositional inverse, apart from counting the permutation polynomials of this type. We further describe other classes of permutation polynomials, which are again closed under compositional inverses.  We also describe the nature of this map and the eigenspaces over prime fields $\mathbb{F}_p$ and show that the generalized eigenspaces contain degree-wise distribution of PPs of $\mathbb{F}_p$. The paper is organized as follows.\\
Section $2$ describes the preliminaries regarding permutation polynomials over finite fields. In Section $3$, we introduce the linear map, its representation, the generalized eigenspaces of the matrix, their dimensions and bases. Section $4$ is about similar maps defined over $\mathbb{F}_q[x]$, their matrices and the similarity of their eigenspaces. We describe the intersection of the generalized eigenspaces of these matrices and their dimension and structure. The case of $\mathbb{F}_{p^2}$ is studied in Section $5$ to demonstrate the utility of studying these linear maps. In this section, we study permutation polynomials of a certain shape, and their compositional inverses. We also enumerate the specific numbers of permutation polynomials that appear in various sub-classes. Some proofs of this section are given in the Appendix. Section $6$ has some concluding remarks.
\section{Preliminaries}
$\mathbb{F}_q$ is the finite field with $q$ elements where $q=p^n$, $p$ being a prime number. Thus the characteristic of the field is $p$. $\mathbb{F}_q^*$ represents the nonzero elements of $\mathbb{F}_q$. A permutation polynomial $f(x)$ of a given finite field is a polynomial that permutes the elements of the field through its evaluation map. This is equivalent to saying that the map $f(x): \mathbb{F}_q \longrightarrow \mathbb{F}_q$ is a bijective map. It is sufficient to consider polynomials having degree $\le q-1$ since two maps $f$ and $g$ are equivalent as far as their evaluation maps are concerned if $f(x) \equiv g(x)~mod~(x^q-x)$ \cite{Mullen:2013:HFF}.  The composition of two polynomials $f(x)$ and $g(x)$ is defined as $f(x) \circ g(x)=f(g(x))~mod~(x^q-x)$. Permutation polynomials of a finite field form a group under composition. \\

It is of interest to know when a given polynomial permutes the given field. We give here a criterion due to Hermite. There are several others like the Raussnitz criteria, Turnwald criteria etc. 

\begin{prop}{\bf Hermite's Criterion:}\label{thm:hermite}
Let $q=p^r$, where $p$ is a prime and $r$ is a positive integer. Then a polynomial $f \in \mathbb{F}_q[x]$ is a permutation polynomial of $\mathbb{F}_q$ if and only if the following two conditions hold:
\begin{enumerate}
\item the reduction of $f(x)^{q-1} \mod (x^q-x)$ is monic of degree $q-1$;
\item for each integer $t$ with $1 \leq t \leq q-2$ and $t \not\equiv 0 \mod p$, the reduction of $f(x)^t \mod (x^q-x)$ has degree $\leq q-2$.
\end{enumerate}
\end{prop}
The Hermite's test  becomes impractical with large $q$.  See \cite{vonzurGathen:1991:tests} for a detailed treatment on this. Hermite's criterion gives some immediate and very useful corollaries. 
\begin{corollary}\label{cor:max_degree}
If $q>2$ and $f(x)$ is a permutation polynomial of $\mathbb{F}_q$ then the reduction of $f\mod x^q-x$ has degree at most $q-2$.
\end{corollary}
\begin{corollary} \label{cor:degree_d}
There is no permutation polynomial of degree d $> 1$  over $\mathbb{F}_q$ if d divides $q-1$.
\end{corollary}
Next, we give the binomial theorem for commutative rings from \cite{Lidl:1986:IFF} for easy reference.
\begin{prop}\label{thm:binomial}
Let $R$ be a commutative ring of characteristic $p$. Then 
\[ 
(a+b)^{p^n}=a^{p^n}+b^{p^n}
\]
for $a,b \in R$ and $n \in \mathbb{N}$.
\end{prop}
With these preliminaries, we now describe a map, its representation and properties in the next section.
\section{A Linear map}
We define ``Permutation Polynomial Representative (PPR)'' as 
\begin{newdef}\label{def:PPR}
A monic permutation polynomial $f(x)$  which maps $0$ to $0$ through its evaluation map is said to be a ``Permutation Polynomial Representative (PPR)".
\end{newdef}
Note that a PPR permutes the elements of $\mathbb{F}_q^*$. For every given PPR $f$ one can construct a family of PPs of the form $\theta_1f+\theta_2$ where $\theta_1 \in \mathbb{F}_{q}^*$ and $\theta_2 \in \mathbb{F}_{q}$. Thus it is enough to study PPRs. Therefore we restrict our attention to polynomials that map $0$ to $0$. \\
Let $V[x]$ be the vector space of all polynomials $f(x) \in \mathbb{F}_q[x]$ of degree $\leq q-2$ such that $f(0)=0$. Note that all possible PPRs belong here. Further note that $\dim(V[x]) = q-2$. We define the linear map 
\[
\phi: V[x] \rightarrow V[x] 
\]
\begin{equation} \label{eq:phidef}
\phi(f(x))=f(x+1)-f(1)
\end{equation}
$\phi$ is a linear map and let $A$ be a matrix representation of $\phi$. Note that $\phi$ or $A$ takes a polynomial in $V[x]$ to another polynomial in $V[x]$ of the same degree. We first make  some simple observations about the map $\phi$. 
\begin{lemma}\label{lem:orderof_phi}
The order of the map $\phi$ is $p$ (the characteristic of the field).
\end{lemma}
\begin{proof}
Let $g(x)=\phi(f(x))=f(x+1)-f(1)$. Therefore $\phi(\phi(f(x))=\phi(g(x))=g(x+1)-g(1)=f(x+2)-f(2)$. Thus 
\[
\phi^2(f)=f(x+2)-f(2)
\]

By induction, $\phi^{k}(f)=f(x+k)-f(k)$. Note that $p~\mod~p=0$ and $f(x)\in V[x]$, so $f(0)=0$; hence we have 
\[
\phi^p(f(x))=f(x)
\]
and therefore order of the map $\phi$ is $p$.
 \end{proof}
 
Thus we can conclude that $A^p = I$. Therefore the characteristic polynomial of $A^p$ is $(\lambda - 1)^{q-2}$. 

\begin{lemma}\label{lem:only_1}
 The only possible eigenvalue of $A$ is $1$.
\end{lemma}
\begin{proof}
 As we have seen, $A^p = I$ and therefore the eigenvalues of $A$ must be the $p$-th roots of unity. In a field of characteristic $p$, the only $p$-th root of unity is $1$ and therefore the only possible eigenvalue of $A$ is $1$. 
\end{proof}

We shall now explore the eigenvectors of $A$ corresponding to this eigenvalue $1$. 

\begin{lemma}\label{lem:monomial_in_ker}
The monomials $x^{p^k}$ over $\mathbb{F}_{p^n}$; for $k=0,1,\ldots, n-1$, are eigenvectors of $A$ corresponding to the eigenvalue $1$.
\end{lemma}
\begin{proof}
Note that $x^{p^n}=x$ over $\mathbb{F}_{p^n}$.
\[
\phi(x^{p^k})=(x+1)^{p^k}-1=x^{p^k}
\]
by \pref{thm:binomial}. Therefore the monomials $x^{p^{k}}$ for $k=0,1,\ldots, n-1$ are eigenvectors of $A$ corresponding to the eigenvalue $1$.
\end{proof}


Are these monomials the only eigenvectors of $A$ ? These are the only monomials which can possibly be eigenvectors of $A$. Clearly, $\phi(x^i)$ where $i$ is not of the form $p^k$ gives a polynomial with multiple terms and therefore cannot be an eigenvector of $A$. Let us now assume that $f(x)$ is a polynomial with multiple terms which is an eigenvector of $A$, that is $\phi(f(x)) = f(x)$. Since sum of two eigenvectors is also an eigenvector, it is therefore enough to assume that the degree of $f(x)$ is not equal to $p^k$ for $k = 0,1,\ldots, n-1$.  

\begin{lemma}
Let $f(x)=x^i+a_{i-1}x^{i-1}+\ldots+a_1x$ be a polynomial in $V[x]$ such that $\phi(f(x)) = f(x)$. Then, $\deg f = i$ is a multiple of $p$.
\end{lemma}
\begin{proof}
\[
\phi(f(x))=(x+1)^i+a_{i-1}(x+1)^{i-1}+\ldots+a_1(x+1)-(1+a_{i-1}+\ldots+a_1)
\]
\[
=x^i+\binom{i}{1}x^{i-1}+\binom{i}{2}x^{i-2}+\ldots+1+a_{i-1}(x^{i-1}+\binom{i-1}{1}x^{i-2}
\]
\[
+\ldots+1)+\ldots+a_1x+a_1-(1+a_{i-1}+\ldots+a_1)
\]
Since $\phi(f(x))=f(x)$ by assumption, comparing the coefficients of the monomial $x^{i-1}$, we get $\binom{i}{1} + a_{i-1} = a_{i-1}$. Thus,
\[
\binom{i}{1}=0 \Rightarrow i\equiv 0~mod~p
\]
Therefore, $i$ is a multiple of $p$.
\end{proof}


We now show that for $A$ the geometric multiplicity of the eigenvalue $1$ is $p^{n-1}$. In other words, there are $p^{n-1}$ linearly independent polynomials that satisfy the equation $\phi(f) = f$. We have already seen that there are $n$ monomials $x^{p^k}$ for $k = 0, 1, \ldots, n-1$ that satisfy the equation $\phi(f) = f$. 

\begin{lemma}\label{lem:first_kernel}
The eigenspace of $A$ corresponding to the eigenvalue $1$ is spanned by the monomials $\langle x, x^p, x^{p^2}, \ldots, x^{p^{n-1}} \rangle$ and polynomials $(x^p-x)^m$ for $m = 2,3, \ldots, (p^{n-1}-1)$ where $m \neq p^k$ for any $k$.
\end{lemma}
\begin{proof} 
That the monomials are eigenvectors of $A$ is clear from \lref{lem:monomial_in_ker}. Let $g(x)=(x^p-x)^m$. 
\[
\phi(g(x))=((x+1)^p-(x+1))^m=(x^p-x)^m=g(x)
\]
due to \pref{thm:binomial}. Therefore the polynomials $(x^p-x)^m$ are eigenvectors of $A$. The set $\langle x, x^p, x^{p^2},\ldots, x^{p^{n-1}},(x^p-x)^2, (x^p-x)^3, \ldots, (x^p-x)^{p^{n-1}-1}\rangle$  are linearly independent too, as they have distinct degrees. 

\end{proof}
Thus, the eigenspace of $A$ has dimension $p^{n-1}$. Observe that another basis for the eigenspace of $A$ is given by $\langle x, (x^p - x), (x^p - x)^2, \ldots , (x^p - x)^{p^{n-1}} \rangle$. The reason for prefering the basis in \lref{lem:first_kernel} is due to the following.   As a first cut, one can look at PPRs that are eigenvectors of $A$. These PPRs have the shape
\begin{equation}\label{eq:first_ker_shape}
f(x)=\sum_{m=2, m \neq p^k}^{p^{n-1}-1} c_m(x^p-x)^m + L(x)
\end{equation}
 where $c_m \in \mathbb{F}_{p^n}$ and $L(x)=\sum_{j=0}^{n-1}d_jx^{p^j}$ for $d_j \in \mathbb{F}_{p^n}$ is a linearized polynomial of $\mathbb{F}_q$. Assuming all $c_m = 0$, all linearized polynomials are eigenvectors of $A$. It is well known that $\mathbb{F}_{p^n}$ can be viewed as a vector space over the base field $\mathbb{F}_p$. Linearized polynomials correspond to linear maps from $\mathbb{F}_{p^n}$ to $\mathbb{F}_{p^n}$ that respect the mentioned vector space structure. Thus, all linearized polynomials that correspond to linear maps of full rank are permutation polynomials. 
 
 Permutation polynomials of the form $(x^p-x+\delta)^s+L(x)$ have been studied in literature \cite{Yuan2008PermutationPO}. Observe that the PPRs corresponding to these permutation polynomials have a shape that matches the shape specified by \eref{eq:first_ker_shape}. Thus these permutation polynomials are related to eigenvectors of $A$.  
 
 Permutation polynomials of the form $A(x) + g(B(x))$ where $A,B$ are linearized polynomials (additive polynomials is the name used in \cite{Akbary08onsome}) and $g(x) \in \mathbb{F}_q[x]$ have also been studied in literature. Note that $x^p - x$ is a linearized polynomial and therefore all polynomials of the shape \eref{eq:first_ker_shape} are of this specific form.

Having characterized the eigenspace of $A$, we now look at the generalized eigenspaces of $A$. We have seen that $\dim (\ker (A - I)) = p^{n-1}$. As $A^p = I$, it is clear that $\dim(\ker (A-I)^p) = \dim(\ker (A^p - I)) = q-2 = p^n - 2$. Therefore, it is enough to look at $\ker (A-I)^k$ for $k = 2, 3, \ldots, p-1$. We now show that $\dim(\ker (A-I)^k) = kp^{n-1}$ for $k = 1, 2, \ldots, p-1$. The case for $k=1$ is already shown.

\begin{lemma}\label{lem:x_times_fx}
If a polynomial $f(x) \in \ker(A-I)^k$ then $g(x)=xf(x)$ is in $\ker(A-I)^{k+1}$ for $k = 1, \ldots, p-1$.
\end{lemma}
\begin{proof}
A basis for $\ker (A-I)$ is given by the monomials $x^{p^i}$ and the polynomials $(x^p - x)^i$, for appropriate values of $i$. We first demonstrate that $x^{p^i + 1}$ and $x(x^p-x)^i$ are in $\ker (A-I)^2$. By evaluation, one gets $\phi(x^{p^i+1}) = (x+1)(x+1)^{p^i} - 1 = x^{p^i + 1} + x^{p^i} + x$ and therefore $(A-I)x^{p^i+1} \in \ker (A-I)$. Similarly, $\phi(x(x^p-x)^i) = (x+1)(x^p-x)^i$ and therefore $(A - I)(x(x^p - x)^i) = (x^p - x)^i \in \ker (A - I)$. Thus for each basis element $f(x)$ in the eigenspace of $A$, we found $xf(x)$ in $\ker (A - I)^2$. As these new polynomials have degrees distinct from the basis vectors of $\ker (A - I)$, therefore one concludes that $\dim(\ker (A-I)^2) = 2p^{n-1}$.  

We now use induction. Consider $\ker(A-I)^k$ which has dimension $kp^{n-1}$. One can form a basis of $\ker (A-I)^k$ which contains $p^{n-1}$ vectors of the form $x^{p^i+k-1}$ and $x^{k-1}(x^p-x)^i$ (for appropriate $i$) which are not contained in $\ker(A-I)^{k-1}$. It is enough to show that $xf(x) \in \ker (A-I)^{k+1}$ for every one of these basis vectors $f(x)$. Observe that  $(x+1)^{p^i + k} - 1 = (x+1)^k(x^{p^i} + 1) - 1$ and once $x^{p^i + k}$ is removed from that expression, all the other monomials are contained in $\ker (A-I)^k$. Thus, $x^{p^i + k} \in \ker (A-I)^{k+1}$. Similarly, $(x+1)^k((x+1)^p - (x+1))^i = (x+1)^k(x^p - x)^i$ and once you remove $x^k(x^p-x)^i$ from that expression, one is left with terms of the form $x^{k-j}(x^p-x)^i$ all of which are contained in $\ker (A-I)^k$.  
\end{proof}

Due to \lref{lem:x_times_fx}, the kernels can be generalized as follows.
\begin{theorem}\label{thm:gen_basis}
$\ker(A-I)^m$ over $\mathbb{F}_{p^n}$ is spanned by $\langle x^j(x^p-x)^i, x^k\rangle$; where $j=0,1,2, \ldots, m-1$; $i=1,2, \ldots, p^{n-1}-1$ and $k=1,2, \ldots, m$.
\end{theorem}
And polynomials in $\ker(A-I)^m$ are of the shape
\begin{equation}\label{eq:gen_kernel_shape}
f(x)=\sum_{i=1,i\neq p^k}^{p^{n-1}-1} \sum_{j=0}^{m-1}c_{ij}x^j(x^p-x)^i+L(x)
\end{equation}
where $L(x)$ is a linearized polynomial of $\mathbb{F}_q$.\\
Literature contains studies on permutation polynomials of the kind $x^jh(x^{\frac{q-1}{d}})$. Notice that such PPRs appear as generalized eigenvectors of $A$, that is they appear in $\ker (A-I)^j$. 

We have discussed the eigenvectors and generalized eigenvectors of the linear map $A$ for extension fields $\mathbb{F}_{p^n}$. Similar results hold for the base field $\mathbb{F}_p$, though with a subtle twist that one should be alert about. In this case, $A$ is a $(p-2) \times (p-2)$ matrix, as $\dim(V[x]) = p-2$. On the other hand, $A^p = I$ as before. Further, as $x^p \equiv x$ on $\mathbb{F}_p$ as far as evaluation map is concerned, therefore $A$ has only one eigenvector, namely $x$. Note that $\alpha x$ (for $\alpha \in \mathbb{F}_p$) are the only linearized polynomials and they are permutation polynomials. 

\tref{thm:gen_basis} has a straightforward corollary over $\mathbb{F}_p$.
\begin{corollary}\label{cor:generalized_kernel_Fp}
$ker(A-I)^m$ over $\mathbb{F}_p$ is spanned by $\langle x, x^2, x^3, \ldots, x^m\rangle$.
\end{corollary}
Thus, all polynomials in $V[x]$ are parsed degree-wise by $ker(A-I)^m$. So, all degree $d$ polynomials in $V[x]$ appear for the first time in $ker(A-I)^d$. Thus, if one were to find a degree distribution of PPRs over a field $\mathbb{F}_p$, then this is equivalent to counting the PPRs present in succesive kernels of $(A-I)^m$. This encourages one to perform a similar exercise on successive kernels of $(A-I)^m$ for the extension fields ($\mathbb{F}_{p^n}$) too.

Before closing this section, let us recall properties of the linear map $\phi$ or its representation matrix $A$. Note that $A$ is a $(q-2) \times (q-2)$ matrix. 
\begin{itemize}
\item Order of the matrix $A$ is $p$. (\lref{lem:orderof_phi})
\item Dimension of $ker(A-I)^p$ is $q-2$, the full space $V[x]$.
\item $Ax^{p^k}=x^{p^k}$ from \lref{lem:monomial_in_ker}. Therefore, $x^{p^k}$ is an eigenvector of $A$.
\item The characteristic polynomial of $A$ is $(\lambda-1)^{q-2}, (q \neq p)$ with $\lambda=1$ being the only characteristic value. 
\item The minimal polynomial of $A$ is $(\lambda-1)^p$ for $(q \neq p)$.
\item The characteristic and minimal polynomials are the same when $q=p$ and that is $(\lambda-1)^{p-2}$.
\item Dimension of eigenspace of $A$ is $p^{n-1}$, that is, geometric multiplicity of the characteristic value $\lambda=1$ is $p^{n-1}$.
\item Dimension of the generalized eigenspaces of $A$ that is $\dim(\ker(A-I)^k)$ is $kp^{n-1}$ for $k=1,2,\ldots, p-1$.
\end{itemize}

In the study of PPRs, it would make sense to classify PPRs in terms of the generalized eigenspaces of the matrix $A$. As we have mentioned, various shapes of permutation polynomials studied in literature can be identified with the specific stage $j$ such that they belong to $\ker (A-I)^j$ but do not belong to $\ker (A-I)^i$ for $i < j$. For the case of $\mathbb{F}_p$, this separation of PPRs by the stage at which they make their first appearance is precisely equivalent to the degree classification of PPRs. So in some sense, $A$ helps us in classifying PPRs for $\mathbb{F}_q$ that resembles the normal degree classification in $\mathbb{F}_p$. The dimension of the eigenspace of $A$ is pretty large ($p^{n-1}$) and therefore it would be helpful if one can find other classifications of the PPRs.

\section{Other `$A$-like' Matrices}
As it turns out, the linear map $A$ is just one among many similar linear maps, having isomorphic properties. We therefore look at a family of linear maps $A_r$ which are isomorphisms of the space of polynomials $V[x]$ introduced in the previous section. This family of isomorphisms can be parametrized by the elements $r \in \mathbb{F}_q$. Let 
\begin{eqnarray*}
A_r & : & V[x] \rightarrow V[x] \\ & & f(x) \rightarrow f(x+r)-f(r)
\end{eqnarray*}
where $r$ is an element of $\mathbb{F}_q$. Observe that $A_0$ is identity map. The map $A$ in the previous section should now be $A_1$. These maps commute : 

\begin{lemma}\label{lem:commuting_matrices}
For $r,s \in \mathbb{F}_q$; $A_r[A_s[f(x)]]=A_s[A_r[f(x)]]= A_{r+s}[f(x)]$. 
\end{lemma}
\begin{proof}
\[
A_s[f(x)]=f(x+s)-f(s)
\]
\begin{eqnarray*}
A_r[A_s[f(x)]] & = & (f(x+r+s)-f(s)) - (f(r+s)-f(s)) \\ & = & f(x+r+s) - f(r+s) \\ & = & (f(x+s+r)-f(r)) - (f(r+s)-f(r)) \\ & = & A_s[A_r[f(x)]]
\end{eqnarray*}
\end{proof}
Thus, matrix multiplication in this family of $(q-2) \times (q-2)$ matrices captures the Abelian structure of addition in $\mathbb{F}_q$. By \lref{lem:commuting_matrices}, $(A_r)^2=A_rA_r=A_{2r}$ and eventually  $(A_r)^p=A_{pr}=A_0=I$. Thus, the cyclic group generated by any $A_r$ has order $p$. Following \lref{lem:only_1}, all $A_r$ have only $1$ as a possible eigenvalue. Following the other lemmas, one can further say that $\dim(\ker (A_r - I)^k) = kp^{n-1}$. It is also clear that the monomials $x^{p^i}$ for $i = 0,1, \ldots, n-1$ are eigenvectors of $A_r$ for any $r \in \mathbb{F}_q$. Thus, every linearized polynomial is an eigenvector of all $A_r$. 

The eigenvectors of $A_r$ which are not monomials may differ depending on $r$. By geometry, the set of $\mathbb{F}_p$-multiples of an element $r$, like $r,2r,\ldots,(p-1)r$ is the ``line'' corresponding to $r$. The number of lines in $\mathbb{F}_q$ is equal to $\ell_q = \frac{p^n-1}{p-1}$. Corresponding to every $r \in \mathbb{F}_q$, one can find $b \in \mathbb{F}_q$ which is a $\ell_q$-th root of unity, such that 
\begin{lemma}\label{lem:line_of_r}
$\ker(x^p-bx)$ is precisely the entire line corresponding to $r$ where $b = r^{p-1}$ is a $\ell_q$-th root of unity.
\end{lemma}
\begin{proof}
$r\in ker(x^p-bx)$ implies $r^p=br$ or $b = r^{p-1}$. For any $\mathbb{F}_p$-multiple of $r$, say, $ir$, we have $((ir)^p-b(ir))=(i^pr^p-ibr)=i(r^p-br)=0$ and therefore, the entire line is in $ker(x^p-bx)$. 
Since $1 = r^{p^n-1} = b^{(p^n-1)/(p-1)}$, $b$ is a $\ell_q$-th root of unity.
\end{proof}
The structure of the eigenspaces of $A_r$ is isomorphic to that of $A$ given in the previous section.
\begin{theorem}\label{thm:kernel_of_A}
The eigenspace of $A_r$ is spanned by the monomials $\langle x, x^p, x^{p^2}, \- \ldots, x^{p^{n-1}}\rangle$ and the polynomials  $(x^p-bx)^i$ for all $i$ from $2,  \ldots, p^{n-1}-1$ where $i$ is not a power of $p$. Here $b = r^{p-1}$ is a $\ell_q$-th root of unity.
\end{theorem}
\begin{proof}
Note that cardinality of the set $\langle x, x^p, x^{p^2},\ldots, x^{p^{n-1}}, (x^p-bx)^i\rangle$ for $i=2,3,\ldots, p^{n-1}-1$ which is not a power of $p$, is indeed $p^{n-1}$ which is equal to the dimension of the eigenspace of $A_r$ ($\dim(\ker(A_r-I)) = p^{n-1}$). Also, the set is linearly independent. What remains to be shown is that $(x^p-bx)^i$ is an eigenvector of $A_r$. 
\[
A_r[(x^p-bx)^i] = ((x+r)^p-b(x+r))^i-(r^p-br)^i = (x^p -bx + r^p - br)^i - (r^p-br)^i
\]
Therefore, $(x^p-bx)^i$ is an eigenvector $A_r$ as $r^p=br$ or $r^{p-1}=b$. 
\end{proof}
Literature abounds in studies of permutation polynomials of the form $(x^p - x)^i + L(x)$, $(x^p - x + \delta)^i + L(x)$ and so on, where $L(x)$ is of the form $\alpha x^p + \beta x$, $\delta \in \mathbb{F}_q$. The PPRs of all of these permutation polynomials appear as eigenvectors of the matrix $A$ from the previous section or $A_1$ as it should be named as per this section. Note that these permutation polynomials are in fact a special case of polynomials in \tref{thm:kernel_of_A}, where $b = 1 = r^{p-1}$ with $r \in \mathbb{F}_p$, that is, they correspond to the specific ``line'' of elements in $\mathbb{F}_p \subset \mathbb{F}_q$. Moreover, if $L(x)$ is permitted to be any linearized polynomial, then these include many more polynomials than those studied in literature for the case where $q = p^n$ and $n > 2$. There are some studies on PPs of this shape for $\mathbb{F}_{q^2}$, but they give $L(x)$ of very specific shapes. In fact, there may be permutation polynomials of the form $(x^p - bx)^i + L(x)$ or $(x^p - bx + \delta)^i + L(x)$ for $b$ which is a $\ell_q$-th root of unity, which have not been captured by the studies in literature. There exists a symmetry in the shapes of PPRs that appear in the eigenspaces of various $A_r$. We argue that this symmetry can be exploited to study a wider class of permutation polynomials. Of course, some of these permutation polynomials may already be included in the studies in literature, especially for the specific case of $\mathbb{F}_{p^2}$.

Note that though there are a total of $q$ matrices $A_r$, one needs to only consider a subset of these matrices. For example, $A_0$ is the identity matrix and therefore gives no information in terms of eigenspaces. Similarly, the eigenstructure of $A_r$ is exactly same as the eigenstructure of $A_{ir}$ where $i \in \mathbb{F}_p^*$.
\begin{lemma}\label{lem:same_kernel}
$\ker(A_r - I)^k = \ker (A_{ir} - I)^k$ for all $r \in \mathbb{F}_q$, $i \in \mathbb{F}_p^*$ and $k = 1,2, \ldots, p$.
\end{lemma}
\begin{proof}

Elementary computation.
\end{proof}
\lref{lem:same_kernel} suggests that it is enough to look at one $A_r$ corresponding to each ``line'' in $\mathbb{F}_q$. Since there are $\ell_q$ lines, it suffices to consider $\ell_q$ matrices $A_r$, to capture all possible variations. Notice that the linearized polynomials are eigenvectors of every one of these matrices $A_r$. So, if one wants to classify PPRs based on these matrices, then it would make sense to not just view the eigenstructure of these matrices but explore common eigenspaces (that is intersections) of these subspaces. However, the number of matrices to consider is $\ell_q$ and this number may be daunting. We now propose a simplification, that captures information without a large computational head.

Let $a$ be a primitive element of $\mathbb{F}_q$.  
Then $1, a, a^2, \ldots , a^{n-1}$ are linearly independent vectors in $\mathbb{F}_q$ when viewed as a vector space over $\mathbb{F}_p$ (here $q = p^n$). We propose to use the matrices 
\[A_1, A_a, A_{a^2},\ldots, A_{a^{n-1}}
\]
Note that any other choice of linearly independent vectors in $\mathbb{F}_q$ yield similar results to what we list below. 
$\ker (A_{a^i}-I)^j$ are isomorphic to each other for each $j$. The PPRs contained in the various subspaces $\ker (A_{a^i}-I)^j$ are similar in shape to those contained in $\ker (A-I)^j$. We now explore the common generalized eigenspaces of the matrices $A_{a^i}$. For this, we define the intersection of these generalized eigenspaces as 
\begin{equation}\label{eq:kernelintersection}
V_k=\bigcap_{i=0}^{n-1} ker(A_{a^i}-I)^k
\end{equation}

Let us now look at $V_1$. This consists of the common eigenvectors of the matrices $A_{a^i}$. Clearly, the monomials $x^{p^i}$ for $i = 0,1,\ldots, n-1$ are all present in $V_1$. As a result, all linearized polynomials are in $V_1$. We have  
\begin{lemma}\label{lem:ker1_dim}
\begin{enumerate}
\item $\dim (V_1) = n$ and is spanned by monomials $\langle x, x^p, x^{p^2}, \ldots, x^{p^{n-1}} \rangle$. 
\item The number of PPRs in $V_1$ is $(p^n-p)(p^n-p^2)\ldots(p^n-p^{n-1})$ 
\item Compositional inverse of every PP in $V_1$ is also in $V_1$ 
\end{enumerate}
\end{lemma}
\begin{proof}
\begin{enumerate}
\item It is clear that $\dim (V_1)$ is at least $n$ since the monomials $x^{p^i}$ are in $V_1$. As all the polynomials $(x^p - bx)^j$ for a particular $j$ are made up of very specific monomials, it suffices to restrict one's attention to the subspace (say $W$) spanned by these monomials. Observe that $\dim (W \cap \ker (A_{a^i}-I) = 1$. Thus the common intersection of $n$ one dimensional subspaces of $W$ is contained in $V_1$. It is easy to see that this common intersection is ${0}$. Therefore $\dim (V_1)$ is precisely $n$.  

\item Consider $n \times n$ matrices with entries from $\mathbb{F}_p$. Carlitz \cite{Carl-found} has established an isomorphism between linearized polynomials in $\mathbb{F}_{p^n}$ and $\mathbb{F}_p^{n \times n}$. Thus every invertible matrix in $\mathbb{F}_p^{n \times n}$ corresponds to a linearized polynomial which is a permutation polynomial. The number of such invertible  matrices is  $(p^n-1)(p^n-p)(p^n-p^2)\ldots(p^n-p^{n-1})$. Therefore the number of PPRs in $V_1$ is  $(p^n-p)(p^n-p^2)\ldots(p^n-p^{n-1})$.
\item Clear from above.
\end{enumerate}
\end{proof}
Experimental evidence suggests that the dimension of $V_k$ is $k^n+n-1$. The case of $k = 1$ is given in \lref{lem:ker1_dim}. We have established that $\dim (V_k) = k^n + n - 1$ for the cases $n = 2, 3$ and small values of $k$. For reasons of brevity, we do not include the proofs here. 

\section{Case of $\mathbb{F}_{p^2}$}
We now specialize to the case of $\mathbb{F}_{p^2}$. The eigenspace of $A_{r}$ for any nonzero $r \in \mathbb{F}_{p^2}$ is spanned by $\langle x, x^p, (x^p-bx)^m\rangle$ for $m=2,3,\ldots, p-1$ and $b = r^{p-1}$ is a $(p+1)$-th root of unity. Every polynomial in this eigenspace is given by 
\[
f(x)= \sum_{m=2}^{p-1}(x^p-bx)^m+L(x)
\]
where $L(x)$ is a linearized polynomial over $\mathbb{F}_{p^2}$. We give technical conditions for a particular sub-class of polynomials in the eigenspaces of the $A_{r}$ matrices to be a PP and establish that their  compositional inverses also belong to the same sub-class. These constructions are parametric and hence, the PPs can be constructed algorithmic-ally. \\

Let $\alpha$ and $\beta \in \mathbb{F}_{p^2}^*$ such that $\alpha^{p+1} \neq \beta^{p+1}$ and $\alpha, \beta$ not simultaneously $0$. Let $b = r^{p-1}$ be a $(p+1)^{th}$-root of unity.  We define 

\begin{newdef}\label{def:parameters}
\[
\gamma=\frac{-\alpha}{\beta^{p+1}-\alpha^{p+1}}
\]
\[
\epsilon=\frac{\beta^p}{\beta^{p+1}-\alpha^{p+1}}
\]
\[
\delta =\frac{-\gamma d -\epsilon}{(\beta^p-\alpha d)^m}
\]
\[
d=(-1)^mb^{mp}
\]
\end{newdef}
Let $\alpha, \beta, b, \gamma, \epsilon$, $\delta$ and $d$ be as in \dref{def:parameters}. Additionally, let $(\beta+b\alpha)^{p-1}=(-1)^mb^{mp-1}$. Then

\begin{theorem}\label{thm:main theorem}
Every polynomial $f(x) \in \mathbb{F}_{p^2}[x]$ of the form
\begin{equation}\label{eq:mainpp}
f(x)=(x^p-bx)^m+\alpha x^p+\beta x
\end{equation}
with $\alpha^{p+1} \neq \beta^{p+1}$ is a PP of $\mathbb{F}_{p^2}$ for $m=2,3,\ldots, p-1$.  The compositional inverse of $f$ is given by 
\[
h(x)=\delta (x^p-dx)^m+\gamma x^p +\epsilon x
\]
\end{theorem}
The theorem is proved in the appendix along with other technical results needed for the proof.\\

We now enumerate the number of PPRs that belong to the sub-class defined in \tref{thm:main theorem}. Note that the following two conditions govern the existence of the permutation polynomials described in \eref{eq:mainpp}.
\begin{equation}\label{eq:condn1}
\alpha^{p+1} \neq \beta^{p+1}
\end{equation}
\begin{equation}\label{eq:condn2}
(\beta+b\alpha)^{p-1}=(-1)^mb^{mp-1}
\end{equation}
It can be established that for every $m$ and every $b \in \mathbb{F}_{p^2}$ which is a $(p+1)$-th root of unity, the number of PPRs where $\alpha$ and $\beta$ satisfy the conditions \eref{eq:condn1} and \eref{eq:condn2} is precisely $p(p-1)^2$. We do not show the actual calculation, due to lack of space.

In the broader class of polynomials having the shape $f(x)=(x^p-bx)^m+\alpha x^p+\beta x$, there are PPRs that do not satisfy the conditions \eref{eq:condn1} and \eref{eq:condn2}. It can be shown that when $m$ is co-prime to $p-1$, or $m = \frac{p+1}{2}$, there are PPRs that do not satisfy \eref{eq:condn1} and \eref{eq:condn2}. We have enumerated the total number of PPRs of the shape $(x^p-bx)^m+\alpha x^p+\beta x$ for every $b  \in \mathbb{F}_{p^2}$ which is a $(p+1)$-th root of unity and $m$ co-prime to $p-1$ to be $p(p-1)(2p-1)$. We can further show that the collection of permutation polynomials for a fixed $m$ and all $b$ which are $(p+1)$-th roots of unity, which are  associated to these extra $p^2(p^2-1)$ PPRs is a sub-class of PPs that are closed under compositional inverses. 

We have further established that for $m = \frac{p+1}{2}$, the number of PPRs of the form $(x^p-bx)^m+\alpha x^p+\beta x$ is even higher than $p(p-1)(2p-1)$ for every $p > 5$. Unfortunately, the exact closed form expression for the number of PPRs in this case have eluded us.

We now explore the spaces $V_k$ that we defined earlier. We do not provide the proofs or actual calculations due to lack of space. Let $a$ be a primitive element of $\mathbb{F}_{p^2}$ and consider the maps $A_1$ and $A_a$. As shown in the previous section, $V_1$ is 2-dimensional and spanned by the monomials $x, x^p$. There are precisely $p(p-1)$ PPRs in $V_1$ given by the linearized polynomials -- $x$ is the unique degree 1 PPR and $p^2 - p - 1$ PPRs $x^p - rx$ for all $r \in \mathbb{F}_{p^2}$ which are not $(p+1)$-th roots of unity. Further, the PPs associated to the PPRs in $V_1$ are closed under compositional inverses (\lref{lem:ker1_dim}).

$V_2$ is always 5-dimensional for $\mathbb{F}_{p^2}$ and is spanned by the monomials $\langle x, x^2, x^p, x^{p+1}, x^{2p} \rangle$. There are no PPRs of degree $2$ and degree $p+1$ due to \cref{cor:degree_d}. Apart from the linearized polynomials already obtained in $V_1$, all other PPRs in $V_2$ have the shape $(x^p - bx)^2 + \alpha x^p + \beta x$ where $b$ is a $(p+1)$-th root of unity. The number of such PPRs are $p(p+1)(p-1)^2$. The PPs here are again closed under compositional inverses. All degree $2p$ PPRs are in $V_2$.

$\dim (V_3) = 10 $ for $\mathbb{F}_{p^2}$ and is spanned by nine monomials $x,x^2,x^3, x^p, \- x^{p+1}, x^{p+2}, x^{2p}, x^{2p+1}, x^{3p}$ and a degree $4p$ polynomial. However, every PPR in $V_3$ lies within the subspace spanned by the nine monomials. It is interesting to note that if we take any two matrices $A_r$ and $A_s$ with $r,s \in \mathbb{F}_{p^2}$ such that $r$ and $s$ are not on the same line, and construct $V_3$ using these matrices, the corresponding basis has a similar structure. The degree $4p$ polynomial that we obtain in each of these cases never generates a PPR. This pushes us towards conjecturing that it is enough to look at subspaces spanned by monomials.

We end this section with the following
\begin{lemma}\label{lem:dim_vk_fpsquare}
Dimension of $V_k$ over $\mathbb{F}_{p^2}$ follow the sequence $2,5,10,17,\ldots$ and is given by $dim(V_k)=k^2+1$ for $k = 1,2, \ldots, p-1$ and $\dim(V_{p}) = p^2-2$.
\end{lemma}

\section{Conclusion}
In this paper, we have introduced a set of linear maps on a specific subspace $V[x] \subset \mathbb{F}_q[x]$. We then demonstrate the shapes of polynomials in the various generalized eigenspaces of these matrices. We propose studying all PPRs by looking at classes of PPRs that appear in specific generalized eigenspaces. The symmetry in the structure of the linear maps is useful, since studying one set helps us understand the other sets. We propose studying PPRs in terms of the intersection space $V_k$ to which they belong. Evidence suggests that this division of PPRs have added advantages like being closed under compositional inverses. Moreover, it also helps in enumeration of the PPRs. Overlaps between different classes of PPRs can be avoided. Moreover, we believe that this approach would prove useful especially to study PPs in $\mathbb{F}_{p^n}$ for $n > 2$.   
\section{Appendix}
The compositional inverse given in \tref{thm:main theorem} needs to be proved. We require to prove some technical conditions for that and we give those in this section. To bring some simplicity in notation, we first define two polynomials as follows.
\begin{newdef}\label{def:gmb polynomial}
The $g_{mb}$ polynomial for $b=a^{i(p-1)}$ is defined as 
\[g_{mb}(x)=(x^p-bx)^m; ~~2 \leq m \leq p-1
\]
\end{newdef}

\begin{lemma}\label{lem:lem1}
\[
g_{mb}(x)^p=(-1)^mb^{mp}g_{mb}(x)
\]
\end{lemma}
\begin{proof}
By \pref{thm:binomial}, $g_{mb}^p=(x-b^px^p)^m$. Since $b^p=b^{-1}$ hence the claim is true. 
\end{proof}
\begin{newdef}\label{def:hmd polynomial}
For every $g_{mb}(x)$ we define a polynomial which we call as the $h_{md}$ polynomial, as follows.
\[
h_{md}(x)=(x^p-dx)^m
\]
 where $d=(-1)^mb^{mp}$ 
\end{newdef}
\begin{lemma}\label{lem:lemma3}
Let $\gamma$ and $\epsilon$ be as in \dref{def:parameters}. 
Then, the following is true.
\begin{enumerate}
\item $\beta^{p+1}-\alpha^{p+1} \in \mathbb{F}_p$.
\item $\gamma\beta^p+\alpha\epsilon=0$ and $\gamma\alpha^p+\beta\epsilon=1$.
\item $\alpha \epsilon^p+\beta \gamma=0$ and $\alpha \gamma^p+\beta \epsilon=1$.
\end{enumerate}
\end{lemma}
\begin{proof}
The proof is straight forward from the definitions in \dref{def:parameters} by mere substitution and \pref{thm:binomial}.

\end{proof}
\begin{lemma}\label{lem:lemma4}
With $\alpha$ and $\beta$ as in \lref{lem:lemma3} and $b$ as in \dref{def:gmb polynomial} we have 
\[
(\beta+b\alpha)^{p-1}=(-1)^mb^{mp-1} \Leftrightarrow \frac{b \beta^p+\alpha^p}{\beta+\alpha b}=(-1)^mb^{mp}
\]
\end{lemma}
\begin{proof}
Let 
\[
\frac{b \beta^p+\alpha^p}{\beta+\alpha b}=(-1)^mb^{mp}
\]
Therefore, we have 
\[
b(\beta^p+b^{-1}\alpha^p)=(-1)^mb^{mp}(\beta +\alpha b)
\]
Note that, $b^{-1}=b^p$. Hence,
\[
(\beta+b \alpha)^p=(-1)^mb^{mp-1}(\beta+\alpha b)
\] as $(\beta^p+b^p\alpha^p)=(\beta+b\alpha)^p$ using the binomial theorem over finite fields of characteristic $p$. Therefore,
\[
(\beta+b\alpha)^{p-1}=(-1)^mb^{mp-1}
\]
The other side can be proved by reversing the arguments.
\end{proof}
\begin{lemma}\label{lem:delta_equi}
If $(\beta+b\alpha)^{p-1}=(-1)^mb^{mp-1}$ then $\delta=\frac{(\beta+b\alpha)^{m-1}}{(\beta^{p+1}-\alpha^{p+1})^m}$.
\end{lemma}
\begin{proof}
Note that $d=(-1)^mb^{mp}$ (\dref{def:hmd polynomial}). And $b^{-1}=b^p$, so using the binomial theorem
\begin{equation}\label{eq:d_def}
(\beta+b\alpha)^{p-1}=(-1)^mb^{mp-1}=db^{-1}\Rightarrow \frac{b\beta^p+\alpha^p}{\beta+b\alpha}=d
\end{equation}
With $\delta$ as in \dref{def:parameters} and $d$ as in \dref{def:hmd polynomial};
\begin{equation}\label{eq:delta_redef}
\delta=\frac{-(\gamma d+\epsilon)}{(\beta^p-\alpha d)^m}
\end{equation}
From \eref{eq:d_def}
\[
\beta^p-\alpha d=\beta^p-\frac{\alpha(b\beta^p+\alpha^p)}{\beta+b\alpha}
\]
With some simplification
\begin{equation}\label{eq:inter1}
\beta^p-\alpha d=\frac{\beta^{p+1}-\alpha^{p+1}}{\beta+b\alpha}
\end{equation}
From \eref{eq:delta_redef} and \eref{eq:inter1} and using
\[
(\gamma d+\epsilon)=\frac{-\alpha d+\beta^p}{\beta^{p+1}-\alpha^{p+1}}=\frac{1}{\beta+b\alpha}
\]
from \eref{eq:inter1}; we have
\[
\delta=\frac{(\beta+b\alpha)^{m-1}}{(\beta^{p+1}-\alpha^{p+1})^m}
\]
\end{proof}
\begin{lemma}\label{lem:deltaproperty}
With $\delta$ as in \dref{def:parameters} $(-1)^mb^{m^2}\delta^{p}=b\delta$ if $(-1)^mb^{mp-1}=(\beta+b\alpha)^{p-1}$.
\end{lemma}
\begin{proof}
Note that 
\[
b/((-1)^mb^{m^2})=(-1)^m(b^{-1})^{m^2-1}=(-1)^m((b^{-1})^{m+1})^{m-1}
\]
From the fact that $b^{-1}=b^p$, $(b^{-1})^{m+1}=(b^p)^{m+1}=b^{mp}b^p=b^{mp}b^{-1}=b^{mp-1}$.
\[
b/((-1)^mb^{m^2})=((-1)^mb^{mp-1})^{m-1}
\]
Since $(-1)^mb^{mp-1}=(\beta+b\alpha)^{p-1}$
\[
b/((-1)^mb^{m^2})=((\beta+b\alpha)^{p-1})^{m-1}=\Big(\frac{(\beta+b\alpha)^p}{\beta+b\alpha}\Big)^{m-1}
\]

\[
=\frac{((\beta+b\alpha)^{m-1})^p}{(\beta^{p+1}-\alpha^{p+1})^{m}}\frac{(\beta^{p+1}-\alpha^{p+1})^{m}}{(\beta+b\alpha)^{m-1}}
\]
Using \lref{lem:lemma3}$(3)$
\[
b/((-1)^mb^{m^2})=\frac{((\beta+b\alpha)^{m-1})^p}{((\beta^{p+1}-\alpha^{p+1})^m)^{p}}\frac{(\beta^{p+1}-\alpha^{p+1})^{m}}{(\beta+b\alpha)^{m-1}}
\]
From \lref{lem:delta_equi}
\[
b/((-1)^mb^{m^2})=\delta^p/\delta
\]
as claimed.
\end{proof}

\begin{lemma}\label{lem:lemma2}
For the polynomial $h_{md}(x)$ in \dref{def:hmd polynomial} we have $h_{md}(x)^p=(-1)^mb^{m^2}h_{md}(x)$.
\end{lemma}
\begin{proof}
\[
h_{md}(x)^p=(x^p-dx)^{mp}=(-1)^md^{mp}(x^p-(d^p)^{-1}x)^m
\]
As $(d^p)^{-1}=((-1)^{mp}b^{mp^2})^{-1}=(-1)^m(b^m)^{-1}=(-1)^{m}(b^{-1})^m$. And $b^{-1}=b^p$, hence 
\[
h_{md}(x)^p=(-1)^m b^{m^2}(x^p-dx)^m=(-1)^mb^{m^2}h_{md}(x)
\]
\end{proof}
\subsection{Proof of \tref{thm:main theorem} }
\begin{proof}
From the definitions of $f(x)$ and $h(x)$ we have 
\begin{equation} \label{eq:eqn1}
h \circ f=\delta (f^p-df)^m+\gamma f^p +\epsilon f
\end{equation}
\[
=\delta (g_{mb}^p+\alpha ^p x+\beta^px^p-dg_{mb}-d\alpha x^p-d\beta x)^m+\gamma f^p+\epsilon f
\]

From \lref{lem:lem1} 
\begin{equation}
h \circ f=\delta(g_{mb}((-1)^mb^{mp}-d)+x^p(\beta^p-\alpha d)-x(\beta d -\alpha^p))^m+\gamma f^p+\epsilon f \nonumber
\end{equation}
Since $d=(-1)^mb^{mp}$ (\dref{def:hmd polynomial}), 
\begin{equation}\label{eq:eqn2}
h \circ f=\delta\Bigg((\beta^p-\alpha d)\Big(x^p-\frac{\beta d-\alpha^p}{\beta^p-\alpha d}x\Big)\Bigg)^m+\gamma f^p+\epsilon f
\end{equation}
From \lref{lem:lemma4}, and \dref{def:hmd polynomial} we have 
\[
\frac{b\beta^p+\alpha^p}{\beta+b\alpha}=d
\]
 Rearranging this we get
\[
b=\frac{\beta d-\alpha^p}{\beta^p-\alpha d} 
\]
therefore \eref{eq:eqn2} can be simplified to 
\[
h \circ f= \delta (\beta^p-\alpha d)^m(x^p-bx)^m+\gamma dg_{mb}+
\]
\[
\gamma \alpha^p x +\gamma \beta^p x^p +\epsilon g_{mb} +\epsilon \alpha x^p+\epsilon \beta x
\]

\[
h \circ f=g_{mb}(\delta (\beta^p-\alpha d)^m+\gamma d+\epsilon) +x^p(\gamma \beta^p+\epsilon \alpha) +x(\gamma \alpha^p+\epsilon\beta)
\]
From \lref{lem:lemma3} $(1)$, and from  \dref{def:hmd polynomial},  $h \circ f=x$.\\
We now consider $f \circ h$ too.\\
\begin{equation} \label{eq:eqn3}
f \circ h= (h^p-bh)^m +\alpha h^p+\beta h
\end{equation}
\[
h^p=\delta^p h_{md}^p+\gamma^p x+\epsilon^px^p
\]
From \lref{lem:lemma2}, $h_{md}^p=(-1)^mb^{m^2}h_{md}$. Therefore,
\[
f \circ h=\Big(\delta^p(-1)^mb^{m^2}h_{md}+\gamma^px+\epsilon^px^p-b\delta h_{md}-b\gamma x^p-b\epsilon x\Big)^m+\alpha h^p+\beta h
\]
From \lref{lem:deltaproperty},
\[
f \circ h=\Big((\epsilon^p-b\gamma)(x^p-\frac{b\epsilon-\gamma^p}{\epsilon^p-b\gamma}\Big)^m+\alpha h^p+\beta h
\]
From \dref{def:parameters}, \dref{def:hmd polynomial} and \lref{lem:lemma4}
\[
\frac{b\epsilon-\gamma^p}{\epsilon^p-b\gamma}=\frac{b\beta^p+\alpha^p}{\beta+b\alpha}=d
\] Hence
\[
f \circ h=(\epsilon^p-b\gamma)^mh_{md}+\alpha h^p+\beta h
\] 
\[
=h_{md}((\epsilon^p-b\gamma)^m+\alpha \delta^p(-1)^mb^{m^2}+\beta\delta)+x^p(\alpha\epsilon^p+\beta\gamma)+x(\alpha\gamma^p+\beta\epsilon)
\]
From \lref{lem:deltaproperty}
\begin{equation}\label{eq:eqn4}
f \circ h=h_{md}((\epsilon^p-b\gamma)^m+\delta(\alpha b+\beta))+x^p(\alpha \epsilon^p+\beta\gamma)+x(\alpha\gamma^p+\beta\epsilon)
\end{equation} 
Note that 
\[
(\epsilon^p-b\gamma)^m=\Big(\frac{\beta+b\alpha}{\beta^{p+1}-\alpha^{p+1}}\Big)^m=\frac{1}{(\beta^p+\alpha d)^m}
\]
from \dref{def:parameters} and \eref{eq:inter1}.  Using definition of $\delta$ and the above 
\begin{equation}\label{eq:eqn5}
(\epsilon^p-b\gamma)^m+\delta(\alpha b+\beta)=\frac{1}{(\beta^p+\alpha d)^m}-\frac{(\gamma d+\epsilon)(\alpha b+\beta)}{(\beta^p+\alpha d)^m}
\end{equation}
Note that
\[
(\gamma d+\epsilon)=\frac{-\alpha d+\beta^p}{\beta^{p+1}-\alpha^{p+1}}=\frac{1}{\beta+b\alpha}
\]
from \eref{eq:inter1}. So right hand side of \eref{eq:eqn5} turns out to be $0$ and from \eref{eq:eqn4}
\[
f \circ h=x^p(\alpha \epsilon^p+\beta\gamma)+x(\alpha\gamma^p+\beta\epsilon)
\]
Therefore, from \lref{lem:lemma3}$(3)$, 
\[
f \circ h=x
\]
and $f(x)$ is a PP of $\mathbb{F}_{p^2}$.
\end{proof}


\end{document}